\documentclass[12pt,a4paper,final]{article}

\usepackage{natbib} 
\usepackage{amsmath}    
\usepackage{amsfonts}   
\usepackage{amssymb}    
\usepackage{amsthm}     
\usepackage{mathrsfs}   
\usepackage{parskip}    
\usepackage{calc}       
\usepackage{enumitem}   
\usepackage{graphicx}   
\usepackage{csquotes}
\usepackage[english]{babel}
\RequirePackage{hyperref}
\hypersetup{
	colorlinks=true, 
	linktoc=all,     
	linkcolor=blue,  
	citecolor=blue,
	urlcolor = blue,
}

\newlength{\vslength}
\newcommand{\ER}{{\it ER}}
\newcommand{\ie}{{\it i.e.}}
\newcommand{\cf}{{\it cf.}}
\newcommand{\eg}{{\it e.g.}}

\newcommand{\EE}{{\mathbb E}}

\newcommand{\RR}{{\mathbb R}}

\newcommand{\scrB}{{\mathscr B}}
\newcommand{\scrC}{{\mathscr C}}

\newcommand{\scrE}{{\mathscr E}}

\newcommand{\scrH}{{\mathscr H}}
\newcommand{\scrI}{{\mathscr I}}

\newcommand{\scrL}{{\mathscr L}}

\newcommand{\scrQ}{{\mathscr Q}}

\newcommand{\scrX}{{\mathscr X}}

\newcommand{\ep}{{\epsilon}}

\newcommand{\ctg}{\mathbin{\lhd}}

\newcommand{\ft}[2]{{\textstyle{\frac{#1}{#2}}}}

\newcommand{\conv}[1]%
{{\mathrel{\,\xrightarrow{\widthof{\,#1\,}}\,}}}
\newcommand{\convas}[1]%
{{\mathrel{\,\xrightarrow{\widthof{\,#1\text{-a.s.}\,}}\,}}}
\newcommand{\convprob}[1]%
{{\mathrel{\,\xrightarrow{\widthof{\,#1\,}}\,}}}
\newcommand{\convweak}[1]%
{{\mathrel{\,\xrightarrow{\widthof{\,#1\text{-w.}\,}}\,}}}

\newcommand{\contig}{\mathop{\vartriangleleft}}

\renewcommand{\qedsymbol}{$\Box$}
\newcommand{\closebox}{{\hfill\qedsymbol}}

\newtheoremstyle{customtheorem}
{0.5em}
{0.2em}
{\itshape}
{}
{\scshape}
{}
{1ex}
{}

\theoremstyle{customtheorem}
\newtheorem{theorem}{Theorem}[section]
\newtheorem{lemma}[theorem]{Lemma}
\newtheorem{proposition}[theorem]{Proposition}
\newtheorem{corollary}[theorem]{Corollary}
\newtheorem{definition}[theorem]{Definition}

\newtheoremstyle{customremark}
{0.5em}
{0.2em}
{}
{}
{\scshape}
{}
{1ex}
{}

\theoremstyle{customremark}
\renewenvironment{proof}{\par\noindent{\scshape Proof}\;}{\closebox\par}
\newtheorem{remark}[theorem]{Remark}
\newtheorem{example}[theorem]{Example}

\newcommand{\comment}[1]{{}}
\newcommand{\extra}[1]{{}}

\begin{document}
\thispagestyle{empty}
\title{\vspace*{-16mm}
	Contiguity and remote contiguity\\of some random graphs
}
\author{
	B.~J.~K.~{Kleijn}${}^{1}$ and S.~Rizzelli${}^{2}$\\[1mm]
	{\small\it ${}^{1}$ Korteweg-de~Vries Institute for Mathematics,
		University of Amsterdam}\\
	{\small\it ${}^{2}$ Department of Statistical Sciences, University of Padova}  
}
\date{\today}
\maketitle
\begin{abstract}\noindent
  Asymptotic properties of random graph sequences, like occurrence
  of a giant component or full connectivity in Erd\H os-R\'enyi graphs,
  are usually derived with very specific choices for defining
  parameters. The question arises to which extent those parameters
  choices may be perturbed, without losing the asymptotic property.
  Writing $(P_n)$ and $(Q_n)$ for two sequences of graph
  distributions, asymptotic equivalence (convergence in
  total-variation) and contiguity ($P_n(A_n)=o(1) \implies
  Q_n(A_n)=o(1)$) have been considered by (Janson, 2010) and
  others; here we use so-called remote contiguity (for some fixed
  $a_n\downarrow 0$, $P_n(A_n)=o(a_n) \implies Q_n(A_n)=o(1)$)
  to show that connectivity properties are preserved in more
  heavily perturbed Erd\H os-R\'enyi graphs. The techniques we
  demonstrate with random graphs here, extend to general
  asymptotic properties, e.g. in more complex large-graph limits,
  scaling limits, large-sample limits, etc.\\[.3em]
  {\bf Keywords} random graphs, contiguity, remote contiguity,
    graph connectivity, giant component\\
  {\bf MSC} 05C80, 05C40, 60B10, 60G30.
\end{abstract}
	

\section{Asymptotic properties of random graphs}
\label{sec:intro}

Most asymptotic properties of random graph sequences are
derived in a highly regular setting, based on certain precise
choices for defining parameters. This specificity raises the
question to which extent perturbations of parameters
leave said properties intact. For perturbations of the
Erd\H os-R\'enyi (\ER) graph \citep{Erdos59}, Janson's
seminal paper \citep{Janson10} discusses the question how
an asymptotic property of one \ER\ graph sequence can be
related to that of another, based on asymptotic equivalence
(convergence in total variation) and contiguity
\citep{LeCam60b,LeCam86,Roussas72,Greenwood85}.

In this paper we propose a more general form of asymptotic
congruence called \emph{remote contiguity}, introduced in \citep{Kleijn21}.
While the conditions inducing asymptotic equivalence or
contiguity can sometimes be too stringent, remote contiguity
is applicable more widely. To demonstrate this, we consider
two sequences of random graphs $(X^n)$ and $(Y^n)$, with
distributions denoted by $(P_n)$ and $(Q_n)$ respectively. The
subscript $n$ denotes the number of vertices in the graph: we look
at graphs that grow to infinite size and study conditions under
which some or all of the asymptotic properties of $(X^n)$ also
apply to $(Y^n)$, based on uniform tightness of re-scaled
log-likelihood ratios of $(Q_n)$ with respect to $(P_n)$.

By an \emph{asymptotic graph property}, we mean any sort of
limit in probability: the property is expressed in the form
$P_n(A_n)\to0$, with some sequence of $n$-vertex graph events
$(A_n)$. In the case at hand we consider well-known connectivity
properties of \ER\ graphs with asymptotically bounded
expected degrees, such as occurrence of a giant component in
the supercritical regime and $O(\log(n))$-fragmentation in
the subcritical regime, as well as asymptotic connectedness
in the less sparse regime of \ER\ graphs with expected degrees
that diverge logarithmically. We also consider the so-called
critical window, in which the largest component displays
asymptotic growth of order $O(n^{2/3})$. (For an extensive
review of random graph asymptotics and asymptotic properties
of \ER\ graphs in particular, see \cite{hofstad16}).

\subsection{Perturbations of graph sequences}

To find laws $Q_n$ for random graphs $Y_n$ that share asymptotic
properties with the $X^n\sim P_n$, one may impose the (sufficient)
condition that the Hellinger (or total-variational) distance between
$P_n$ and $Q_n$ goes to zero in the limit (asymptotic equivalence, see
\cite{Janson10}, definition 1.1): then any property of $(X^n)$ is shared by the
graph sequence $(Y^n)$ in the sense that, for any sequence of
$n$-vertex graph events $(A_n)$,
\[
P_n(A_n) - Q_n(A_n)\to0.
\] 
Asymptotic equivalence occurs if and only if there exists a
coupling of $X_n$ and $Y_n$ such that $P(X_n\neq Y_n)$ tends
to zero as $n\to\infty$ (see \cite{Janson10}, theorem~4.2). 

Janson argues that \emph{asymptotic
	equivalence is too strong} as a condition for the sharing of
asymptotic properties: for example, the giant component occurs
in a large family of inhomogeneous versions of the \ER\
graph \citep{Bollobas07}, much larger than the subset of all
asymptotically equivalent graphs. Le~Cam's notion of contiguity
\citep{LeCam60b,LeCam86,LeCam90} is a weaker, more appropriate condition
for the sharing of asymptotic properties: $(Q_n)$ is said to be
\emph{contiguous with respect to $(P_n)$} (notation
$Q_n\contig P_n$) if
\begin{equation}
	\label{eq:defcontiguity}
	P_n(A_n)=o(1) \quad\Rightarrow\quad Q_n(A_n)=o(1),
\end{equation}
for any sequence of $n$-vertex events $(A_n)$. Janson
applies contiguity to sequences of perturbed \ER\ graphs and
demonstrates its wider applicability. (We discuss Janson's
condition for contiguity of inhomogeneous \ER\ graphs in
section~\ref{sec:IERRC}.)

The main point of this paper is that \emph{contiguity is
	still too strong}, if the $P_n(A_n)$ are known to converge
to zero faster than a certain rate $(a_n)$ (rather than
just being $o(1)$). Given a sequence
$a_n\downarrow0$, we say that $(Q_n)$ is
\emph{$a_n$-remotely contiguous with respect to $(P_n)$}
(notation $Q_n\contig a_n^{-1}P_n$) if
\begin{equation}
	\label{eq:defrc}
	P_n(A_n) = o(a_n) \quad\Rightarrow\quad Q_n(A_n)=o(1),
\end{equation}
for any sequence of $n$-vertex events $(A_n)$. Remote
contiguity was introduced in \citep{Kleijn21} for the frequentist
analysis of Bayesian, posterior-based limits, and argued to
offer generalization of contiguity-based statistical arguments
(which typically apply in smooth-parametric (\eg\
local-asymptotically normal, \citep{LeCam60b}) models, to
much more general (\eg\ non-parametric) setting; see
\citep{Kleijn21}, subsection~3.3. In extreme-value theory,
\citep{Falk20} and \citep{Padoan22} use remote contiguity
to generalize a consistency conclusion reached for an idealized
sequence of data-distributions, to the general class of
sequences that are realistic for the data in the problem.

To use remote contiguity with sequences of random graphs, we
look for sequential models $(P_n)$ for $n$-vertex graphs with
asymptotic properties $(A_n)$ and \emph{known} $(a_n)$, as in
(\ref{eq:defrc}), and analyse the family of those sequential graph
distributions $(Q_n)$ that satisfy $Q_n\contig a_n^{-1}P_n$,
to conclude that random graphs distributed according to $(Q_n)$
share the asymptotic property reflected by the events $(A_n)$.
It is noted that for many asymptotic graph properties,
\emph{sharp} rates $(a_n)$ are known (see, \eg, \citep{hofstad16}). 

In section~\ref{sec:connER} we briefly review \ER\ random graphs
to fix notation and we recall Janson's condition for contiguity
of \ER\ graph sequences. In section~\ref{sec:IERRC} we apply remote
contiguity to sequences of inhomogeneous \ER\ graphs, compare with
Janson's condition and formulate a weaker, Lindeberg-type
condition that involves the rate sequence $(a_n)$ (which, in most
cases, is not only sufficient but also necessary).
In section~\ref{sec:ERconnectivity} 
we define so-called \emph{remotely contiguous domains of attraction}
for all aforementioned asymptotic connectivity properties of \ER\ graphs.
Section~\ref{sec:concdisc} summarizes conclusions and discusses possible
directions of further research. Details on remote contiguity are
discussed in \citep{Kleijn21} and summarized in appendix~\ref{app:rc}.

The ultimate goal of this paper is to convince the reader
that remote contiguity provides a meaningful generalization of
the notion of contiguity, applicable to a much wider range of
problems. We remark that, like asymptotic equivalence and
contiguity, remote contiguity compares \emph{any} pair of
sequences of probability distributions (including (but not
limited to) the comparison of \ER\ graph distributions).



\section{Homogeneous and inhomogeneous \ER\ graphs}
\label{sec:connER}

Let $G_n=([n],E_n)$ denote the complete graph with $n$ vertices, with
vertex labels from $[n]:=\{1,2,\ldots,n\}$ and edge set $E_n$ (which
does not include self-loops; the edge between vertices $i$ and
$j$ is denoted $(ij)$). Denote the space of all sub-graphs of $G_n$
by $\scrX_n$. The homogeneous \ER\ random graph with
edge-probability $0<p<1$, a random element of $\scrX_n$
denoted $X^n$, is $([n],E'_n)$ where $E'_n$ contains any $e\in E_n$
independently with probability $p$. The presence or absence of an
edge $e=(ij)$ from $E_n$ in the graph $X^n$ is expressed in terms
of (independent) random variables, $X^n_{ij}=1$ or $X^n_{ij}=0$
respectively. The degree of vertex $i$ is denoted by
$D_i(X^n)$. We denote the distribution of $X^n$ with
edge-probability $p$, by $P_{p,n}$. When we speak of an ``\ER\
graph (with edge probability $p$)'' below, we refer to the
\emph{sequence of distributions} $(P_{p,n})$, as an element
of $\prod_n M^1(\scrX_n)$, and we denote the class of all \ER\
graphs as $\scrE$. We generalize from $\scrE$ in two stages: 
we distinguish the class $\scrH\subset\prod_n M^1(\scrX_n)$ of
all \emph{homogeneous \ER\ graphs}, containing all sequences $(Q_n)$ of
the form $Q_n=P_{p_n,n}$, for some $n$-dependent $0\leq p_n\leq1$;
and we generalize further by considering the class
$\scrI\subset\prod_n M^1(\scrX_n)$ of all \emph{inhomogeneous \ER\
	graphs}, in which the probability for occurrence of an edge may
depend on the vertices it connects (see below).

For the asymptotic properties of homogeneous \ER\ graphs,
the $n$-dependence of the $(p_n)$ plays a central role:
prime example is the sequence of \ER\ graphs with edge probabilities
$p_n=\lambda/n$ ($\lambda<1$, $\lambda=1$ and $\lambda>1$
characterize the so-called \emph{subcritical, critical and
	supercritical regimes}). With a slight abuse of notation, we denote the distributions of these
graphs with $P_{\lambda,n}$. In some cases, we also leave room for
further $n$-dependence, \eg\ \ER\ graphs $Y_n\sim P_{\lambda_n,n}$.

The inhomogeneous \ER\ random graph with edge-probabilities
$0<p_{n,ij}<1$, ($1\leq i < j \leq n$), is
$([n],F_n)$ where $F_n$ contains any $e=(ij)\in E_n$ independently
with probabilities $p_{n,ij}$; we denote the distribution of the
inhomogeneous \ER\ graph $Y^n$ with edge-probabilities
$(p_{n,ij}:1\leq i < j \leq n)$, by $P_{(p_{n,ij}),n}$. The presence or
absence of an edge $e=(ij)$ from $E_n$ in $F_n$ is expressed
in terms of (independent) random variables, $Y^n_{ij}=1$ or
$Y^n_{ij}=0$ respectively; write $Y^n_{ij}\sim P_{p,n,ij}$ and
$P_{(p_{n,ij}),n}=\prod_{i<j}P_{p,n,ij}$. 
Of foremost interest to this work are edge probabilities of the form
$p_{n,ij}=\mu_{n,ij/n}$, for  which it is assumed throughout that 
\begin{equation}\label{eq:limitbound}
	\lim_{n \to \infty} \sup_{i<j} \frac{\mu_{n,ij}}{n}  <1,
\end{equation}
\ie\ that edge probabilities are uniformly bounded away from one
(compare with the bound $p_{ij,n}<0.9$ in Janson's theorem,
thm.~\ref{thm:janson10} below). For later reference, we define
the parameter spaces $\Lambda_n=\RR^{\ft12n(n-1)}$ for all $n\geq1$,
with inner-product norm,
\[
\|\lambda_n-\mu_n\|_{2,n}^2
=\sum_{i<j}(\lambda_{n,ij}-\mu_{n,ij})^2.
\]

In what follows, we shall examine to which extent known properties
of \ER\ graphs, like the occurrence of a \emph{giant component} or
a \emph{critical window}, are shared in the classes of homogeneous
and inhomogeneous \ER\ graphs (like the \emph{stochastic block model}
that is central in network theory). For a general review, see
\citep{Bollobas07}; for more in relation to contiguity, see
\citep{Janson10}, examples~3.1,~3.5 and~3.6; for other possibilities,
see his remark~1.6 and \citep{Janson95}.
For the following theorem and lemma, consider two inhomogeneous
\ER\ graphs distributed according to $P_n=P_{(p_{n,ij}),n}$
and $Q_n=P_{(q_{n,ij}),n}$. Restricted to contiguity and transcribed
into our notation, Janson's corollary~2.12 says the following.
\begin{theorem}
	\label{thm:janson10}
	{\citep{Janson10}} Assume that $\sup_{i<j}p_{n,ij}<0.9$. If,
	\begin{equation}
		\label{eq:janson10}
		\sum_{i,j}\frac{(p_{n,ij}-q_{n,ij})^2}{p_{n,ij}}=O(1),
	\end{equation}
	then $Q_n\contig P_n$.
\end{theorem}
(Janson's corollary~2.12 also specifies that $P_n$ and $Q_n$
are asymptotically equivalent if the sum in (\ref{eq:janson10}) is $o(1)$
rather than $O(1)$ but that fact plays no role in what follows.)


\section{Remote contiguity of inhomogeneous \ER\ graphs}
\label{sec:IERRC}

The application of remote contiguity to the \ER\ graph involves
sufficient conditions (considered in subsection~\ref{sub:IERRCsuff}), and
necessary conditions (considered in subsection~\ref{sub:IERRCnecc}).

\subsection{Sufficient conditions for remote contiguity of \ER\ graphs}
\label{sub:IERRCsuff}

To extend the results of \citep{Janson10} to remotely
contiguous random graphs, consider the likelihood-ratio with observation
$Y^n$, which is that of $\ft12n(n-1)$ independent Bernoulli experiments:
\begin{equation}
	\label{eq:likratioIERG}
	\begin{split}
		\frac{dP_n}{dQ_n}(Y^n)
		&= \prod_{1\leq i<j\leq n}
		\Bigl(\frac{p_{n,ij}}{q_{n,ij}}\Bigr)^{Y^n_{ij}}
		\Bigl(\frac{1-p_{n,ij}}{1-q_{n,ij}}\Bigr)^{1-Y^n_{ij}}\\
		&= \prod_{1\leq i<j\leq n}
		\Bigl(\frac{p_{n,ij}}{1-p_{n,ij}}
		\frac{1-q_{n,ij}}{q_{n,ij}}\Bigr)^{Y_{n,ij}}
		\Bigl( \frac{1-p_{n,ij}}{1-q_{n,ij}} \Bigr)\\
		&= \exp\Bigl(-\sum_{i<j}\bigl(k_{n,ij}Y^n_{ij} + l_{n,ij}\bigr)\Bigr),
	\end{split}
\end{equation}
where,
\[
k_{n,ij} = \log\Bigl(\frac{q_{n,ij}(1-p_{n,ij})}{p_{n,ij}(1-q_{n,ij})}\Bigr),
\quad
l_{n,ij} = \log\Bigl( \frac{1-q_{n,ij}}{1-p_{n,ij}} \Bigr).
\]
Before stating the lemma, we note that the Kullback-Leibler divergences
of $P_n$ with respect to $Q_n$ equal
\begin{equation}
	\label{eq:QnPnKLdiv}
	-\EE_{Q_n}\log\frac{dP_n}{dQ_n}
	= \sum_{i<j} \bigl(k_{n,ij}q_{n,ij}+l_{n,ij}\bigr)\geq 0.
\end{equation}

\begin{lemma}
	\label{lem:UT}
	If we write $\Delta_n=-\EE_{Q_n}\log(dP_n/dQ_n)+\log(a_n)$, and
	for every $\ep>0$ there is an $M>0$ such that, for large enough
	$n$,
	\begin{equation}
		\label{eq:uniftightunnormalized}
		Q_n\biggl(
		\sum_{i<j}k_{n,ij}(Y^n_{ij}-q_{n,ij})> \log(M)-\Delta_n
		\biggr)<\ep,
	\end{equation}
	then $Q_n\contig a^{-1} P_n$.
\end{lemma}
\begin{proof}
	Consider the application of lemma~\ref{lem:rcfirstlemma}-{\it (ii)} to
	the likelihood ratios (\ref{eq:likratioIERG}), with $Y^n\sim Q_n$: for
	given $\ep>0$, let $M>0$ be as in (\ref{eq:uniftightunnormalized}),
	then,
	\[
	Q_n\biggl(
	\sum_{i<j}\bigl(k_{n,ij}Y^n_{ij} + l_{n,ij}\bigr)+\log(a_n) > \log(M)
	\biggr) < \ep,
	\]
	which we may re-write to
	\[
	Q_n\biggl( a_n \Bigl( \frac{dP_n}{dQ_n}(Y^n) \Bigr)^{-1} > M \biggr)
	< \ep,
	\]
	for large enough $n\geq1$.
\end{proof}

Clearly of primary concern is the connection with \citep{Janson10}:
the following proposition illustrates how Janson's corollary~2.12
(represented in abridged form in theorem~\ref{thm:janson10})
relates to remote contiguity.
\begin{proposition}
	Assume that $\sup_{i<j} p_{n,ij} < C/(1+C)$, for some $C>0$ and that
	Janson's condition~(\ref{eq:janson10}) holds. Then,
	\[
	Q_n \contig a^{-1} P_n,
	\]
	for any $a_n\to0$. 
\end{proposition}
\begin{remark}
	The assertion of the above proposition, $Q_n \contig a^{-1} P_n$
	for \emph{any rate $a_n$}, is equivalent to contiguity,
	$Q_n\contig P_n$, the assertion of theorem~\ref{thm:janson10}.
\end{remark}
\begin{proof}
	According to (\ref{eq:uniftightunnormalized}), remote contiguity
	revolves around control of tail probabilities for the sequence
	$\sum_{i<j} k_{n,ij}(Y_{ij}^n -q_{n,ij})$. A sufficient condition
	for uniform tightness is
	\[
	\EE_{Q_n}\biggl(\sum_{i<j} k_{n,ij}(Y_{ij}^n-q_{n,ij})\biggr)^2=O(1),
	\]
	which we prove below. Note that
	\begin{equation}
		\EE_{Q_n}\biggl(\sum_{i<j} k_{n,ij}(Y_{ij}^n -q_{n,ij})\biggr)^2
		= \sum_{i<j} k_{n,ij}^2q_{n,ij}(1-q_{n,ij})
	\end{equation}
	and that
	\[
	\begin{split}
		k_{n,ij}^2 &\leq
		4 \max \biggl\{\Bigl| \log\Bigl(\frac{q_{n,ij}}{p_{n,ij}}\Bigr) \Bigr|,
		\Bigl| \log\Bigl(\frac{1-q_{n,ij}}{1-p_{n,ij}}\Bigr) \Bigr|\biggr\}^2\\[2mm]
		&\leq 4\biggl( \log\Bigl(\frac{q_{n,ij}}{p_{n,ij}}\Bigr) \biggr)^2
		+ 4 \biggl( \log\Bigl(\frac{1-q_{n,ij}}{1-p_{n,ij}}\Bigr) \biggr)^2
	\end{split}
	\]
	and using the inequality $(\log(y))^2 \leq (y-1)^2/y$ either with
	$y= q_{n,ij}/p_{n,ij}$ or with $y=(1-q_{n,ij})/(1-p_{n,ij})$, we
	obtain
	\[
	\begin{split}
		k_{n,ij}^2 &\leq 4 \Bigl(\frac{q_{n,ij}}{p_{n,ij}} -1 \Bigr)^2
		\frac{p_{n,ij}}{q_{n,ij}}
		+ 4 \Bigl(\frac{1-q_{n,ij}}{1-p_{n,ij}} -1 \Bigr)^2
		\frac{1-p_{n,ij}}{1-q_{n,ij}}\\[2mm]
		&= 4\frac{(p_{n,ij} -q_{n,ij} )^2 }{p_{n,ij}} \frac{1}{q_{n,ij}}
		+ 4\frac{(p_{n,ij} -q_{n,ij} )^2 }{1-p_{n,ij}} \frac{1}{1-q_{n,ij}}.
	\end{split}
	\]
	Substituting and using that $\sup_{i<j} p_{n,ij} < C/(1+C)$, we find
	\[
	\begin{split}
		\sum_{i<j} k_{n,ij}^2 q_{n,ij}(1-q_{n,ij})
		&\leq 4 \sum_{i<j} \frac{(p_{n,ij} -q_{n,ij} )^2 }{p_{n,ij}} (1-q_{n,ij})
		+ 4 \sum_{i<j} \frac{(p_{n,ij} -q_{n,ij} )^2 }{1-p_{n,ij}} q_{n,ij}\\
		&\leq 4(1+C) \sum_{i<j} \frac{(p_{n,ij} -q_{n,ij} )^2 }{p_{n,ij}}
	\end{split}
	\]
	and under Janson's condition (\ref{eq:janson10}) the term on the
	right-hand side is $O(1)$. In addition, in view of the inequality
	$\log(y) \leq (y-1)$, it holds that
	\[
	\begin{split}
		-\EE_{Q_n}\log\frac{dP_n}{dQ_n}
		&= \sum_{i<j} \log\Bigl(\frac{q_{n,ij}}{p_{n,ij}} \Bigr) q_{n,ij}
		+ \sum_{i<j} \log\Bigl(\frac{1-q_{n,ij}}{1-p_{n,ij}} \Bigr)(1-q_{n,ij})\\
		&\leq \sum_{i<j} \Bigl(\frac{q_{n,ij}}{p_{n,ij}}-1\Bigr)q_{n,ij}
		+ \sum_{i<j} \Bigl(\frac{1-q_{n,ij}}{1-p_{n,ij}}-1\Bigr)(1-q_{n,ij})\\
		&= \sum_{i<j} \frac{(q_{n,ij} -p_{n,ij})^2}{p_{n,ij}(1-p_{n,ij})}
		\leq (1+C) \sum_{i<j} \frac{(q_{n,ij} -p_{n,ij})^2}{p_{n,ij}}.
	\end{split}
	\]
	Under Janson's condition (\ref{eq:janson10}) the term on the
	right-hand side is $O(1)$. Consequently $-\Delta_n\to \infty$
	\emph{for any $a_n \to 0$}. Together with uniform tightness of
	the sequence of sums $\sum_{i<j}(k_{n,ij}Y_{ij}^n +l_{n,ij})$,
	this leads to remote contiguity $Q_n\contig a_n^{-1}P_n$
	for any $a_n\to 0$.
\end{proof}

The above contiguity proof, however, does not exploit the presence
of a sum of independent components to the full extent. To sharpen
the argument, one normalizes the sums appropriately and imposes
sufficient conditions for remote contiguity. In this case the
contributions to sums are independent, and the appropriate
normalization constants are
\[
s_n^2=\sum_{i<j}k^2_{n,ij}q_{n,ij}(1-q_{n,ij})
\]
for all $n\geq1$, as per Lindeberg's theorem. To illustrate how
conditions for remote contiguity weaken those for contiguity, we note
that Janson's condition (\ref{eq:janson10}) implies that $s_n^2$
remains bounded, whereas below in lemma~\ref{lem:IERRC} $s_n^2$
may diverge.
\begin{lemma}
	\label{lem:IERRC}
	Assume that $s_n<\infty$ for every $n\geq1$, and $s_n\to\infty$.
	Suppose that, for every $\ep>0$,
	\begin{equation}
		\label{eq:ERLindeberg}
		\frac{1}{s_n^2}\sum_{i<j}
		\EE_{Q_n}\bigl( k_{n,ij}^2(Y^n_{ij}-q_{n,ij})^2
		1_{\{k_{n,ij}|Y^n_{ij}-q_{n,ij}| > \ep s_n\}} \bigr) \to 0.
	\end{equation}
	Then $Q_n\contig a_n^{-1} P_n$ for any $(a_n)$, $a_n\downarrow0$,
	such that
	\begin{equation}
		\label{eq:RCIERG}
		\frac1{s_n}\Bigl(\sum_{i<j}
		\bigl(k_{n,ij}q_{n,ij}+l_{n,ij}\bigr)+\log(a_n)\Bigr)\to-\infty.
	\end{equation}
\end{lemma}
\begin{proof}
	Apply the Lindeberg-Feller condition of theorem
	\ref{thm:lindebergfeller} to (\ref{eq:uniftightunnormalized}):
	$s_n$-normalized sums converge weakly to the standard normal
	distribution, if (\ref{eq:ERLindeberg}) holds. Note that, under
	condition (\ref{eq:RCIERG}), we have $-\Delta_n/s_n\to\infty$, and we
	conclude that, for every $\ep>0$ and any choice for $M>0$,
	condition (\ref{eq:uniftightunnormalized}) holds for large enough $n$.
\end{proof}

\subsection{Necessary conditions for remote contiguity of \ER\ graphs}
\label{sub:IERRCnecc}

Next we argue that, when \eqref{eq:ERLindeberg} holds true and $s_n \to \infty$,
then \eqref{eq:RCIERG} is also a necessary condition for remote contiguity
at rate $a_n$.
\begin{lemma}
	\label{lem:necessary}
	Assume that $s_n<\infty$ for every $n\geq1$, and $s_n\to\infty$ and
	suppose that \eqref{eq:ERLindeberg} is satisfied for every $\ep>0$.
	Then \eqref{eq:RCIERG} is necessary for $Q_n\contig a_n^{-1} P_n$
	to hold. 
\end{lemma}
\begin{proof} For every $n\geq1$, let $\nu_n$ be a measure that dominates
	both $P_n$ and $Q_n$ (\eg\ $\nu_n= (P_n +Q_n)/2$) and define
	$p_n=dP_n/d\nu_n$ and $q_n=dQ_n/d\nu_n$. Let $(a_n),(b_n)$ such
	that $a_n,b_n>0$, $a_n,b_n\downarrow 0$ such that $s_n^{-1}\log (b_n) \to 0$.
	Define
	\begin{equation}
		C_n = \left\lbrace \, y^n \in \scrX_n:\,
		q_n(y^n) > \frac{1}{a_n b_n}p_n(y^n)
		\right\rbrace.
	\end{equation}
	It is immediate that
	\begin{equation}
		P_n(C_n) 
		\leq \int_{C_n} a_n b_n q_n(y^n) d \nu_n(y^n)
		\leq a_n b_n Q_n(C_n)=o(a_n).
	\end{equation}
	For any $\eta>0$ and $n$ large enough, we find
	\begin{equation}
		\begin{split}
			Q(C_n)&=Q_n\biggl( a_n b_n
			  \Bigl(\frac{dP_n}{dQ_n}(Y_n)\Bigr)^{-1}>1\biggr)\\[2mm]
			&= Q_n \left( \frac{\sum_{i<j} k_{n,ij}(Y_{ij}^n -q_{n,ij})}{s_n}
			  > \frac{-\log (b_n)}{s_n} -\frac{\Delta_n}{s_n} \right)\\[2mm]
			& \geq Q_n \left( \frac{\sum_{i<j} k_{n,ij}(Y_{ij}^n -q_{n,ij})}{s_n}
			  > \eta -\frac{\Delta_n}{s_n}\right),
		\end{split}
	\end{equation}
	since $-\log (b_n) /s_n \to 0$. If (\ref{eq:RCIERG}) does not hold,
	$\lim_{n\to \infty} \Delta_n/s_n>-\infty$, so that, for some $M>0$
	and all $n$ large enough,
	\begin{equation}
		Q_n(C_n) \geq Q_n \left(
		  \frac{\sum_{i<j} k_{n,ij}(Y_{ij}^n -q_{n,ij})}{s_n} > \eta + M
		\right).
	\end{equation}
	By theorem~\ref{thm:lindebergfeller},
	$\sum_{i<j} k_{n,ij}(Y_{ij}^n -q_{n,ij})/s_n$ converges weakly to a standard
	normal distribution. Then $\liminf_{n \to \infty} Q_n(C_n)>0$ and $(Q_n)$ is
	not $a_n$-remotely contiguous with respect to $(P_n)$. Conclude that
	(\ref{eq:RCIERG}) is necessary for remote contiguity at rate $(a_n)$.
\end{proof}

To estimate whether sequential data $(Y_n)$, $Y_n\in\scrX_n$ was generated
by $(P_n)$ or by $(Q_n)$, statisticians use (randomized tests based on)
test functions $\phi_n:\scrX_n\to[0,1]$. A test function is considered
(minimax-)optimal if it minimizes the sum of type-I and type-II errors:
\[
\pi_n(\phi_n) = \EE_{P_n}\phi_n(Y_n) + \EE_{Q_n}(1-\phi_n(Y_n)).
\]
As it turns out, there is a general upper bound for $\pi_n$ in terms
of the Hellinger affinity $\alpha(P_n,Q_n)$ between $P_n$ and $Q_n$:
\[
\inf_\phi \pi_n(\phi) \leq \int_{\scrX_n}
\sqrt{p_n(y_n)q_n(y_n)}\,d\nu_n(y_n)\quad(=:\alpha(P_n,Q_n)),
\]
and the likelihood-ratio test function minimizes $\pi_n$ (see
also \cite{LeCam86}, section~16.4). Clearly, if there exists a
sequence $(\phi_n)$ such that $\pi_n(\phi_n)=o(1)$, then $(P_n)$
and $(Q_n)$ are not contiguous.

To reason likewise regarding $a_n$-remote contiguity, consider an
$a_n$-weighted version of $\pi_n$,
\[
\pi'_n(\phi_n) = a_n^{-1}\EE_{P_n}\phi_n(Y_n) + \EE_{Q_n}(1-\phi_n(Y_n)).
\]
Reasoning the same as in the contiguous case, we find the following
correspondence between testability and remote contiguity.
\begin{lemma}
	\label{lem:testnorc}
	If there exists a sequence $(\phi_n)$ such that $\pi'_n(\phi_n)=o(1)$,
	then the sequences $(P_n)$ and $(Q_n)$ are not $a_n$-remotely
	contiguous. This is the case whenever $\alpha(P_n,Q_n)=o(a_n^{1/2})$.
\end{lemma}
\begin{proof}
	For every $n\geq1$ the likelihood ratio test function, defined
	for all $y_n\in\scrX_n$ by
	\[
	\psi_n(y_n) = 1_{\{q_n>a_n\,p_n\}}(y_n),
	\]
	minimizes $\pi'_n$. Suppose that there exists a sequence $(\phi_n)$
	such that $\pi'_n(\phi_n)=o(1)$. Then $\pi'_n(\psi_n)=o(1)$, so
	there are events $A_n=\{y_n:q_n(y_n)>a_n\,p_n(y_n)\}$ such that
	$P_n(A_n)=o(a_n)$, but $Q_n(A_n)\to1$, showing that $(P_n)$ and
	$(Q_n)$ are not $a_n$-remotely contiguous. We note the following
	upper bound for $\pi'_n$,
	\[
	\begin{split}
		\pi'_n(\psi_n) &= \inf_{\phi}\pi'_n(\phi)
		= \int_{\{q_n>a_n\,p_n\}} p_n(y_n)\,d\nu_n(y_n)
		+ \int_{\{q_n\leq a_n\,p_n\}} q_n(y_n)\,d\nu_n(y_n)\\
		&\leq \int_{\{q_n>a_n\,p_n\}} \sqrt{a_n^{-1}\,p_n(y_n)q_n(y_n)}\,d\nu_n(y_n)
		+ \int_{\{q_n\leq a_n\,p_n\}} \sqrt{a_n\,p_n(y_n)q_n(y_n)} \,d\nu_n\\
		&\leq a_n^{-1/2}\,\alpha(P_n,Q_n),\phantom{\int}
	\end{split}
	\]
	where the last bound holds for large enough $n$.
\end{proof}
In the proof of lemma~\ref{lem:necessary} we change the argument of
lemma~\ref{lem:testnorc} slightly (through the inclusion of a separate
sequence $b_n\downarrow0$), but the essence is the same: the existence
of certain test sequences precludes remote contiguity.

In the case of $n$-vertex \ER\ graph distributions, the Hellinger
affinity is equal to the product of Hellinger affinities
for each of the independent, Bernoulli-distributed random variables
$Y_{n,ij}$:
\begin{equation}
	\label{eq:hellaffER}
	\alpha(P_n,Q_n)
	= \prod_{i<j} \Bigl( \sqrt{p_{n,ij}\phantom{(}q_{n,ij}}
	+\sqrt{(1-p_{n,ij})(1-q_{n,ij})} \Bigr).
\end{equation}

\subsection{Perturbations of \ER\ graphs}
\label{sub:ApplRC} 

Lemma~\ref{lem:IERRC} formulates a condition that delimits the range
of applicability for remote contiguity in terms of a Lindeberg-type
condition involving the sequence of Kullback-Leibler divergences. In
this section, we simplify that condition with sufficient
conditions formulated directly in terms of the defining parameters
of the \ER\ graphs.
\begin{lemma}
	\label{lem:ER}
	Choose $q_{n,ij}=\lambda_{n,ij}/n$, $p_{n,ij}=\mu_{n,ij}/n$ with
	$0\leq\lambda_{n,ij},\mu_{n,ij}\leq n$ and define
	$P_n=P_{(p_{n,ij}),n}$, $Q_n=P_{(q_{n,ij}),n}$, for all $n\geq1$
	and all $1\leq i,j\leq n$. Assume that
	\begin{equation}
		\label{eq:hp}
		r_n := \sup_{i<j} \frac{\left|\mu_{n,ij}-\lambda_{n,ij}\right|}{
			\mu_{n,ij}
		} \to 0, \quad 
		R_n := \sum_{i<j}\frac{(\lambda_{n,ij}-\mu_{n,ij})^2}{\mu_{n,ij}
			(n-\mu_{n,ij})} \to \infty.
	\end{equation}
	Then, $Q_n\contig a_n^{-1}P_n$ if and only if
	$a_n=o(\exp(-R_n))$. If instead $R_n=O(1)$, then
	$Q_n\contig P_n$.
\end{lemma}
\begin{proof}
	Let $n\geq1$ be given. Using the inequality $ \log (1+x) \leq x$,
	valid for any $x>-1$, we estimate the KL divergence as follows
	\[
	\begin{split}
		-\EE_{Q_n}&\log\frac{dP_n}{dQ_n} (Y^n)
		= \sum_{i<j} \log\Bigl(\frac{q_{n,ij}}{p_{n,ij}} \Bigr) q_{n,ij}
		+ \sum_{i<j} \log\Bigl(\frac{1-q_{n,ij}}{1-p_{n,ij}} \Bigr)(1-q_{n,ij})\\
		&=\sum_{i<j} \biggl( \frac{\lambda_{n,ij}}{n}
		\log\Bigl(1+\frac{\lambda_{n,ij}-\mu_{n,ij}}{\mu_{n,ij}} \Bigr)\\
		&\qquad\qquad
		+ \sum_{i<j} \log\Bigl(1
		-\frac{(\lambda_{n,ij}-\mu_{n,ij})/n}{1-\mu_{n,ij}/n} \Bigr)
		\Bigl(1-\frac{\lambda_{n,ij}}{n}\Bigr) \biggr)\\
		&\leq 
		\sum_{i<j} \frac{\lambda_{n,ij}-\mu_{n,ij}}{n}
		\Bigl( \frac{\lambda_{n,ij}}{\mu_{n,ij}}
		- \frac{1-\lambda_{n,ij}/n}{1-\mu_{n,ij}/n} \Bigr)\\
		&= \sum_{i<j}
		\Bigl( 1+\frac{\mu_{n,ij}/n}{1-\mu_{n,ij}/n} \Bigr)
		\frac{(\lambda_{n,ij}-\mu_{n,ij})^2}{n\mu_{n,ij}}
		= \sum_{i<j}\frac{(\lambda_{n,ij}-\mu_{n,ij})^2}{\mu_{n,ij}(n-\mu_{n,ij})}.
	\end{split}
	\]
	Note that, by \eqref{eq:limitbound}, \eqref{eq:hp} and the expansion $\log (1+x)=x+O(x^2)$,
	($x\to0$), we have
	\[
	\begin{split}
		k^2_{n,ij} &= \biggl(\log\Bigl(\frac{\lambda_{n,ij}}{\mu_{n,ij}}
		\frac{1-\mu_{n,ij}/n}{1-\lambda_{n,ij}/n}\Bigr)\biggr)^2\\[2mm]
		&= \biggl(\log\Bigl(1+\frac{\lambda_{n,ij}-\mu_{n,ij}}{\mu_{n,ij}}\Bigr)
		- \log\Bigl(
		1-\frac{(\lambda_{n,ij}-\mu_{n,ij})/n}{1-\mu_{n,ij}/n}\Bigr)\biggr)^2\\[2mm]
		&= \Bigl(\frac{\lambda_{n,ij}-\mu_{n,ij}}{\mu_{n,ij}}\Bigr)^2
		\Bigl( \frac1{1-\mu_{n,ij}/n} + O(r_n)\Bigr)^2,
	\end{split}
	\]
	for all $1\leq i<j\leq n$, and
	\[
	\begin{split}
		q_{n,ij}(1-q_{n,ij})
		&= \frac{\lambda_{n,ij}}{n}
		\Bigl(1-\frac{\lambda_{n,ij}}{n}\Bigr)\\[2mm]
		&= \frac{\mu_{n,ij}}{n}
		\biggl(1+\frac{\lambda_{n,ij}-\mu_{n,ij}}{\mu_{n,ij}}\biggr)
		\biggl(1-\frac{\mu_{n,ij}}{n}
		\Bigl(1+\frac{\lambda_{n,ij}-\mu_{n,ij}}{\mu_{n,ij}}\Bigr)\biggr)\\[2mm]
		&= \frac{\mu_{n,ij}}{n}
		\Bigl( 1 
		-\frac{\mu_{n,ij}}{n}+O(r_n)
		\Bigr).
	\end{split}
	\]
	As a consequence of the above two displays, we have
	\[
	\begin{split}
		s_n^2 &= \sum_{i<j}k^2_{n,ij}q_{n,ij}(1-q_{n,ij})\\
		& = \sum_{i<j} \Bigl(\frac{\lambda_{n,ij}-\mu_{n,ij}}{\mu_{n,ij}}\Bigr)^2
		\biggl( \frac1{1-\mu_{n,ij}/n} + O(r_n) \biggr)^2
		\frac{\mu_{n,ij}}{n}
		\Bigl( 1 
		-\frac{\mu_{n,ij}}{n}+O(r_n)
		\Bigr)\\[2mm]
		&= \sum_{i<j} \frac{(\lambda_{n,ij}-\mu_{n,ij})^2}{\mu_{n,ij} (n -\mu_{n,ij}) } 
		\left( 1+\left(1-\frac{\mu_{n,ij}}{n}\right) O(r_n)\right)^2 
		\Bigl( {1 + O(r_n)} 
		\Bigr)\\[2mm]
		&=(1+o(1))\sum_{i<j} \frac{(\lambda_{n,ij}-\mu_{n,ij})^2}{\mu_{n,ij}
			(n -\mu_{n,ij}) }.
	\end{split}
	\]
	Condition~(\ref{eq:RCIERG}) is then satisfied for any $a_n$ such that 
	$\limsup_{n \to \infty } R_n^{-1}\log (a_n) < -1$.
	It remains to verify condition~(\ref{eq:ERLindeberg}). To that end,
	consider, for any $\ep,\delta>0$,
	\[
	\begin{split}
		\frac{1}{s_n^2}\sum_{i<j}
		\EE_{Q_n}&\bigl( k_{n,ij}^2(Y^n_{ij}-q_{n,ij})^2
		1_{\{k_{n,ij}|Y^n_{ij}-q_{n,ij}| > \ep s_n\}} \bigr)\\
		&\leq \frac{1}{\ep^\delta s_n^{2+\delta}}\sum_{i<j}
		\EE_{Q_n}|k_{n,ij}|^{2+\delta}|Y^n_{ij}-q_{n,ij}|^{2+\delta}\\
		& = \frac{1}{\ep^\delta s_n^{2+\delta}}\sum_{i<j}
		|k_{n,ij}|^{2+\delta} q_{n,ij}(1-q_{n,ij})
		\left( (1-q_{n,ij})^{1+\delta} + q_{n,ij}^{1+\delta}\right)\\
		& \leq \frac{1}{\ep^\delta s_n^{2+\delta}}\| k_n \|_{\infty,n}^\delta
		\sum_{i<j} k_{n,ij}^2 q_{n,ij}(1-q_{n,ij})\\
		&=  \frac1{\ep^\delta}\biggl(\frac{\| k_n \|_{\infty,n}}{s_n}\biggr)^\delta
		\leq O\biggl(\frac{r_n}{R_n^{1/2}}\biggr)^\delta \to 0,
	\end{split}
	\]
	implying that condition \eqref{eq:ERLindeberg} is satisfied. All the
	assumptions of lemma \ref{lem:IERRC} are thus fulfilled and an
	application of the latter, along with lemma \ref{lem:necessary},
	yields the first result. 
	
	If we replace the hypothesis $R_n \to \infty$ with $R_n=O(1)$, we have
	\[
	-\EE_{Q_n}\log\frac{dP_n}{dQ_n}(Y^n) = O(R_n) = O(1).
	\]
	The hypothesis $R_n=O(1)$ also implies that  $s_n^2 = (1+o(1))R_n =O(1)$
	and, in turn, uniform tightness of the sequence
	$\sum_{i<j}k_{n,ij}(Y_{ij}^n-q_{n,ij})$, so that $Q_n\contig P_n$.
\end{proof}

\begin{remark}
	The uniform convergence assumption $r_n \to 0$ in \eqref{eq:hp} is
	not strictly necessary to have $a_n^{-1}P_n \triangleright Q_n$
	for all rate sequences $a_n=o(\exp(-R_n))$.  At the cost of a
	more involved proof, the results of lemma \ref{lem:ER} can be
	extended to more general cases, such as that where $r_n$ is
	suitably bounded from above but not necessarily decaying to
	$0$. However, assuming $r_n \to 0$ does not appear overly
	restrictive, since it does not preclude, for example, increasing
	variance of log-likelihood-ratios $\log(dP_n/dQ_n)$ (unlike
	\eqref{eq:janson10}) and diverging $L_p$ distances
	$\| \mu_n -\lambda_n\|_{n,p}$, with $1 \leq p<\infty$.
\end{remark}

The following results give examples of applications of the
previous lemma to inhomogeneous and homogeneous perturbations
of a homogeneous \ER\ graph.

\begin{corollary}{\it (Homogeneous perturbation of the \ER\ graph)}\\
	\label{cor:homper}
	Choose $q_{n,ij}=\lambda_{n}/n$, $p_{n,ij}=\lambda/n$ with
	$0\leq\lambda_{n}\leq n$, $1\leq i,j\leq n$, and define
	$P_n=P_{(p_{n,ij}),n}$, $Q_n=P_{(q_{n,ij}),n}$, for all $n\geq1$.
	Assume that $\lambda_n\to\lambda$. If
	$n(\lambda -\lambda_n)^2 \to \infty$, then: $Q_n \contig a_n^{-1} P_n$,
	if and only if,
	\[
	a_n=o\left(\exp\left(-\frac{n}{2\lambda}(\lambda_n -\lambda)^2\right)\right). 
	\]
	If instead $n(\lambda_n-\lambda)^2=O(1)$, then $Q_n \contig P_n$.
\end{corollary}
\begin{proof}
	We have $r_n=|\lambda_n-\lambda|/\lambda$ and, as $n \to \infty$, 
	\[
	R_n = \binom{n}{2} \frac{(\lambda_n - \lambda)^2}{\lambda(n-\lambda)}
	= \frac12 (1+O(n^{-1})) \frac{n (\lambda_n -\lambda)^2}{\lambda}.
	\]
	Then the result follows immediately from lemma \ref{lem:ER}.
\end{proof}
\begin{corollary}{\it (Inhomogeneous perturbation of the \ER\ graph)}\\
	\label{cor:inh}
	Choose $q_{n,ij}=\lambda_{n,ij}/n$, $p_{n,ij}=\lambda/n$ with
	$0\leq\lambda_{n}\leq n$,  $1\leq i,j\leq n$, and define
	$P_n=P_{(p_{n,ij}),n}$, $Q_n=P_{(q_{n,ij}),n}$, for all $n\geq1$. Assume that
	\[
	\bigl\|\lambda_n-\lambda\bigr\|_{\infty,n}=o(1).
	\]
	If also $n^{-1}\| \lambda_n - \lambda \|_{2,n}^2 \to \infty$, 
	then $Q_n\contig a_n^{-1}P_n$, if and only if
	\[
	a_n= o\left(
	\exp\left(- \frac{\| \lambda_n - \lambda \|_{2,n}^2}{n \lambda}\right)
	\right).
	\]
	If instead $n^{-1}\| \lambda_n - \lambda \|_{2,n}^2 =O(1)$,
	then $Q_n\contig P_n$.
\end{corollary}
\begin{proof}
	The result follows immediately from lemma \ref{lem:ER}, when we note that
	$r_n=\| \lambda_n - \lambda \|_{n,\infty}/\lambda$ and that
	\[
	R_n = \frac{ \| \lambda_n - \lambda \|_{2,n}^2} {\lambda(n-\lambda)}
	= (1+O(n^{-1}))\frac{\| \lambda_n - \lambda \|_{2,n}^2}{n \lambda}
	\]
	as $n \to \infty$.
\end{proof}

\section{Connectivity properties of \ER\ graphs}
\label{sec:ERconnectivity}

In this section, we collect several well-known properties of \ER\
graphs and examine their so-called \emph{remotely contiguous
	domains of attraction}, the families of perturbed homogeneous and
inhomogeneous \ER\ graphs that maintain connectivity
properties like the occurrence of a giant component,
$n^{2/3}$-scaling of the giant component at criticality,
$O(\log(n))$-fragmentation or full asymptotic connectedness
through remote contiguity.

\subsection{The giant component in the supercritical \ER\ graph}

As was first shown in \citep{Erdos59}, every supercritical \ER\
graph contains a \emph{giant component}, \ie\ a sequence of connected
components in $X^n$ containing a non-vanishing fraction of all
vertices, with probability growing to one. More precisely, we have
the following theorem.
\begin{theorem}
	\label{thm:ERsupercrit}
	For every $\lambda>1, \nu\in(1/2,1)$ there is a
	$\delta(\lambda,\nu)>0$ such that,
	\begin{equation}
		\label{eq:ERsupercritgiantcomponent}
		P_{\lambda,n}\Bigl(
		\bigl||\scrC_{\text{max}}|-\zeta_\lambda n \bigr| > n^{\nu} \Bigr)
		= O\bigl(n^{-\delta(\lambda,\nu)}\bigr),
	\end{equation}
	where $\zeta_\lambda$ is the survival probability of a Poisson branching
	process with mean offspring $\lambda$.
\end{theorem}
For a proof of this classical result (and the specific way in which
$\delta(\lambda,\nu)$ depends on $\lambda$ and $\nu$) see, for example,
(\cite{hofstad16}, theorem~4.8).

To generalize the occurrence of a giant component,
we pose conditions for sequences of laws $(Q_n)$ of \ER\ graphs
such that, for some $0<\delta<\delta(\lambda,\nu)$,
\begin{equation}
	\label{eq:IERRC}
	Q_n \contig n^{\delta} P_{\lambda,n}.
\end{equation}
Let $\scrQ(\lambda,\delta)$ denote the collection of all sequences
$(Q_n)$ in $\prod_n M^1(\scrX_n)$ satisfying (\ref{eq:IERRC}). In
all those cases, a giant component containing asymptotic fraction
$\zeta_\lambda$ of all vertices occurs (to within order $n^\nu$
vertices). For given $\lambda>1$ and $\nu\in(1/2,1)$, the
sequences $(Q_n)$ for which (\ref{eq:IERRC}) holds for some
$0<\delta<\delta(\lambda,\nu)$ is the union
\[
\scrQ(\lambda,\nu)=\bigcup\bigl\{
\scrQ(\lambda,\delta):0<\delta<\delta(\lambda,\nu)\bigr\}
\]
and, for given $\lambda>1$, the union over all $\nu$ contains
the cases in which a giant component containing a fraction
$\zeta_\lambda$ occurs (to within some negligible
$n^\nu$-fraction of the vertices). We call $\scrQ(\lambda)=
\cup\{\scrQ(\lambda,\nu):\nu\in(1/2,1)\}$ the \emph{remotely
	contiguous domain of attraction for the occurrence of a giant
	component containing $\zeta_\lambda n$ vertices}. Ultimately,
the union $\scrQ=\cup\{\scrQ(\lambda):\lambda>1\}$ forms a
class in which a giant component (containing {\it some}
asymptotically non-vanishing fraction of all vertices) will
form with probability {growing to} one. We refer to that class as the
\emph{remotely contiguous domain of attraction for
	the occurrence of a giant component}.

\begin{example}\label{ex:HomPerSuper}
	{\it (Homogeneous Perturbation of supercritical \ER\ graphs)}\\
	For some $\lambda>1$ choose $P_n=P_{\lambda,n}$ and $Q_n=P_{\lambda_n,n}$,
	with $\lambda_n\to\lambda$. For any choice $1/2<\nu<1$, let
	$0<\delta<\delta(\lambda,\nu)$ as in theorem~\ref{thm:ERsupercrit}
	be given. 
	
	As we have seen in corollary~\ref{cor:homper},
	$R_n=\ft{n}{2}\lambda^{-1}(1+O(n^{-1}))(\lambda_n -\lambda)^2$, so to
	render the rate $a_n$ in definition~(\ref{eq:defrc}) high enough to
	cover the probabilities for occurrence of a giant component,
	\cf\ theorem~\ref{thm:ERsupercrit}, we choose
	\begin{equation}
		\label{eq:ERsupercritperturbed}
		\lambda_n = \biggr(
		1 + \sqrt{\frac{2\delta}{\lambda}\frac{\log(n)}{n}} \biggl)\lambda,
	\end{equation}
	so that $R_n = \bigl( 1+ O(n^{-1})\bigr)\delta\log(n)$.
	Then
	\[
	n (\lambda_n - \lambda)^2 \to \infty,
	\quad\text{and}\quad
	n^{-\delta(\lambda,\nu)}=o(\exp(-R_n)).
	\]
	Therefore, by corollary~\ref{cor:homper}, we can conclude that, for all
	homogeneous $\lambda_n$-perturbations of the \ER\ graph sequence of
	the form (\ref{eq:ERsupercritperturbed}), a giant component containing
	an asymptotic fraction $\zeta_\lambda$ of the vertices occurs
	to within order-$n^\nu$ vertices,
	\[
	P_{\lambda_n,n}\Bigl(
	\bigl||\scrC_{\text{max}}|-\zeta_\lambda n \bigr| > n^{\nu} \Bigr)
	= o(1),
	\]
	since (\ref{eq:ERsupercritgiantcomponent}) ensures the occurrence
	of such a giant component at $\lambda>1$. We may therefore characterize
	the class of homogeneous \ER\ graphs in the remotely contiguous domain of
	attraction for the occurrence of a giant component containing
	$\zeta_\lambda n$ vertices, as
	\[
	\scrH \cap \scrQ(\lambda)
	= \bigcup_{\nu\in(1/2,1)} \biggl\{
	(P_{\lambda_n,n})\in \scrH\,:\,
	(\lambda_n-\lambda)^2
	< 2\lambda\,\delta(\lambda,\nu){\frac{\log(n)}{n}, (n\geq1)}
	\biggr\},
	\]
	and the class of homogeneous \ER\ graphs in the remotely contiguous
	domain of attraction for the occurrence of a giant component as
	\[
	\scrH \cap \scrQ
	= \bigcup \bigl\{\scrH \cap \scrQ(\lambda):\lambda>1
	\bigr \}.
	\]
\end{example}
\begin{example}{\it (Inhomogeneous Perturbation of supercritical
		\ER\ graphs)}\\
	Denote by $\scrI_{\infty}$ the class of inhomogeneous \ER\
	graphs obtained via \emph{uniform} perturbations of a homogeneous
	\ER\ graph $P_{\lambda,n}$, with $\lambda>0$, \ie\ the class of
	\ER\ graphs with edge probabilities $q_{n,ij}=\lambda_{n,ij}/n$
	satisfying $\sup_{i<j}|\lambda_{n,ij}-\lambda|\to0$ for some
	$\lambda>0$. In an analogous fashion, resorting to corollary
	\ref{cor:inh}, we can characterize the class of uniformly perturbed
	inhomogeneous \ER\ graphs in $\scrQ(\lambda)$, with
	$\lambda>1$, as follows,
	\[
	\scrI_\infty \cap \scrQ(\lambda)
	= \bigcup_{\nu\in(1/2,1)}\biggl\{
	(P_{(\lambda_{n,ij}),n})\in\scrI_\infty\,:\,
	\sum_{i<j}(\lambda_{n,ij}- \lambda)^2
	< \lambda\,\delta(\lambda,\nu)\,n\log(n),(n\geq1)\biggr\},
	\]
	and the class of uniformly perturbed inhomogeneous \ER\ graphs
	in the remotely contiguous domain of attraction for the
	occurrence of a giant component as
	\[
	\scrI_\infty \cap \scrQ
	= \bigcup \bigl\{\scrI_\infty \cap \scrQ(\lambda):\lambda>1
	\bigr \}.
	\]
\end{example}

\subsection{Fragmentation in subcritical \ER\ graphs}

Define $I_\lambda=\lambda-1-\log(\lambda)$ for $0<\lambda<1$, that is,
for the subcritical regime of the \ER\ graph. 
\begin{theorem}
	\label{thm:ERsubcrit}
	For given $0<\lambda<1$ and every $a>I_\lambda^{-1}$, there exists a
	$\delta=\delta(a,\lambda)>0$ such that
	\[
	P_{\lambda,n}\Bigl(
	|\scrC_{\text{max}}| \geq a\log(n)\Bigr)
	=O(n^{-\delta}).
	\]
	{Moreover, for any $a < I_\lambda^{-1}$, there exists} an $\eta=\eta(a,\lambda)>0$ such that
	\[
	P_{\lambda,n}\Bigl(
	|\scrC_{\text{max}}| \leq a\log(n)\Bigr)
	=O(n^{-\eta}).
	\]
\end{theorem}
For a proof of the above, see, for example, \citep{hofstad16},
theorems~4.4-4.5.  To prove that the largest connected component in
other random graph sequences has cardinality lying between two multiples of
$\log(n)$, we require again (\ref{eq:IERRC}) for some
$0<\delta< \min(\delta(a,\lambda), \eta(a', \lambda))=:\zeta(\lambda, a, a')$
and $0<a < I_\lambda^{-1} < a'$. We thus define the remotely contiguous domain
of attraction for fragmentation into clusters of maximal cardinality
$I_{\lambda}^{-1}\log(n)$,
\[
\scrQ(\lambda,a,a')=\bigcup\bigl\{
\scrQ(\lambda,\delta):0<\delta<\zeta(\lambda,a, a')\bigr\}.
\]
The union over all $0<\lambda<1$ and $0<a<I_\lambda^{-1}<a'<\infty$
forms the remotely contiguous domain of attraction for fragmentation
into clusters of maximal cardinality of order $\log(n)$:
\[
\scrL = \bigcup
\bigl\{\scrQ(\lambda,a,a')\,:\,0<\lambda <1,\,
0<a<I_\lambda^{-1}<a'<\infty\bigr\}.
\]
\begin{example}{\it (Homogeneous Perturbation of subcritical \ER\ graphs)}\\
	Following reasoning similar to that of example \ref{ex:HomPerSuper}
	and applying corollary~\ref{cor:homper}, we characterise the class
	of \ER\ graphs with maximal connected component of order $\log(n)$,
	which are obtained by homogeneous perturbations of a subcritical
	graphs with $0<\lambda<1$, as,
	\[
	\begin{split}
		\scrL(\lambda) 
		=\bigcup \biggl\{ (P_{\lambda_n,n})\in&\scrH \,:\,
		0<a<I_\lambda^{-1}<a'<\infty,\,\\
		&(\lambda_n-\lambda)^2 < 2\lambda\zeta(\lambda,a,a')
		\frac{\log(n)}{n},\,(n\geq1) \biggr\}.
	\end{split}
	\]
	And we define the homogeneous part of the remotely contiguous domain
	of attraction for fragmentation into clusters of maximal cardinality
	of order $\log(n)$, by $\scrL\cap\scrH=\cup_{0<\lambda<1}\scrL(\lambda)$. 
\end{example}

\begin{remark}
	The above example can be extended to inhomogeneous perturbations
	by application of corollary \ref{cor:inh}. We leave the details to
	the reader.
\end{remark}

\subsection{Maximal connected components in the critical \ER\ graph}
\label{sec:ERcritical}

It is well-known that the largest connected components in
a sequence of \ER\ graphs at criticality ($\lambda=1$) have
cardinalities of order $O(n^{2/3})$. In fact there exists a
so-called \emph{critical window} of $O(n^{-1/3})$
homogeneous perturbations around $\lambda=1$, for which
this critical behaviour of the largest connected component
remains valid.
\begin{theorem}
	\label{thm:ERcrit}
	For some $\theta\in\RR$, every $n\geq1$ and all $1\leq i<j\leq n$,
	define $\lambda_{n,ij}=\lambda_n=1+\theta n^{-1/3}$. There exists a
	constant $b=b(\theta)$ such that,
	\[
	P_{\lambda_n,n}\Bigl(
	a\, n^{2/3} \leq |\scrC_{\text{max}}| \leq a^{-1}\, n^{2/3}\Bigr)
	\geq 1-b\, a,
	\]
	for all $a<1$.
\end{theorem}
For a proof of the above, see, for example, \citep{hofstad16}, theorem~5.1.

Below, we examine to which extent the remotely contiguous domain of
attraction for occurrence of a maximal connected component of order
(approximating) $n^{2/3}$ around the critical point $\lambda=1$,
coincides with the perturbations of order $n^{-1/3}$ in the parameter
$\lambda$ that theorem~\ref{thm:ERcrit} guarantees.

To re-formulate the question: for some $\lambda_n\to1$,
define the homogeneous \ER\ graphs $Y_n$ distributed
according to $Q_n=P_{\lambda_n,n}$, $P_n=P_{1,n}$ and
analyse the requirement,
\[
Q_n\contig \omega_n P_n,
\] 
for any rate $a_n=1/\omega_n$.

To render the assertion of theorem~\ref{thm:ERcrit} at $\lambda=1$
amenable to extension by remote contiguity, we have to make a
choice for a sequence $a_n\to0$: applied to $\lambda=1$,
theorem~\ref{thm:ERcrit} guarantees that there exists a constant
$b=b(0)>0$ such that
\begin{equation}
	\label{eq:ERcritmaxcluster}
	P_{1,n}\Bigl( |\scrC_{\text{max}}| < a_n n^{2/3}\,\,
	\text{or} \,\, |\scrC_{\text{max}}| > a_n^{-1} n^{2/3} \Bigr)
	\leq {b}\,a_n.
\end{equation}
We examine the family of perturbed \ER\ graphs that displays the
same $a_n$-adjusted critical maximal cluster size of order
$n^{2/3}$.

The choice for $(a_n)$ is of great influence on the maximal
permitted perturbation $|\lambda_n-1|$.
\begin{lemma}
	\label{lem:aff}
	Let $\lambda_n\to1$ as $n\to\infty$, such that $\lambda_n-1=O(n^{-1/3})$,
	then there exists a constant $A>0$, such that:
	\[
	\alpha(P_{\lambda_n,n}, P_{1,n})
	=A\exp\left(-\ft1{16}n(\lambda_n-1)^2\right)+o(1).
	\]
\end{lemma}
\begin{proof}
	For $Q_n=P_{\lambda_n,n}$, $P_n=P_{{1},n}$, the Hellinger
	affinity, \cf\ equation~(\ref{eq:hellaffER}), is given by,
	\[
	\begin{split}
		\alpha(P_n,Q_n)
		&= \biggl( \frac{\lambda_n^{1/2}}{n}
		+\Bigl(1-\frac{1}{n}({\lambda_n}+1)+\frac{\lambda_n}{n^2}\Bigr)^{1/2}
		\biggr)^{\binom{n}{2}}\\
		&= \biggl( 1+\frac{(1+(\lambda_n-1))^{1/2}}{n}
		-\frac{\lambda_n+1}{2n}+ O(n^{-2})
		\biggr)^{\binom{n}{2}}\\
		&= \biggl( 1+\frac{1+\ft12(\lambda_n-1)
			-\ft18(\lambda_n-1)^2}{n}
		-\frac{\lambda_n+1}{2n}+O(n^{-1}(\lambda_n-1)^3)+O(n^{-2})
		\biggr)^{\binom{n}{2}}\\
		&= \biggl( 1-\frac{(\lambda_n-1)^2}{8n}
		+O(n^{-1}(\lambda_n-1)^3)+O(n^{-2})
		\biggr)^{\binom{n}{2}}\\
		&= \exp\Bigl(-\ft1{16}n(\lambda_n-1)^2
		+O(n(\lambda_n-1)^3)+O(1)\Bigr)\phantom{\biggl\{}\\
		&= A\exp\bigl(-\ft1{16}n(\lambda_n-1)^2\bigr)+o(1),\phantom{\biggl\{}\\
	\end{split}
	\]
	for some constant $A>0$.
\end{proof}
A slightly more detailed version of the above proof shows that, whenever
$\lambda_n-1=o(n^{-1/2})$, the representation of $\alpha(P_{\lambda_n,n}, P_{1,n})$  in Lemma \ref{lem:aff} holds with $A=1$ and therefore $P_{\lambda_n,1} \contig \,P_{1,n}$, in which case remote contiguity applies
with any $a_n$ decaying to $0$, \eg\ for some small $\ep>0$ and
the choice $a_n=n^{-\ep}$, we find that
\begin{equation}
	\label{eq:RCcritCmax}
	P_{\lambda_n,n}\Bigl( |\scrC_{\text{max}}| < n^{2/3-\ep}\,\,
	\text{or} \,\, |\scrC_{\text{max}}| > n^{2/3+\ep} \Bigr)\to0.
\end{equation}
On the other hand, in light of lemma~\ref{lem:necessary}, if
$n|\lambda_n-1|^2$ goes to $\infty$ fast enough, that is, if
$\alpha_n(P_{\lambda_n,1}, P_{1,n})=o(a_n^{1/2})$,
$(P_{\lambda_n})$ is not $a_n$-remotely contiguous with respect
to $(P_{1,n})$. For example, for some small $\ep>0$ and
the choice $a_n=n^{-\ep}$, we find that, if
\[
\liminf_{n\to\infty}\frac{n(\lambda_n-1)^2}{\log(n)} > 8\ep,
\]
then $(P_{\lambda_n,n})$ is not $n^{-\ep}$-remotely contiguous
with respect to $(P_{1,n})$. An application of corollary~\ref{cor:homper}
allows to further refine the requirement by imposing
\[
a_n =o\left(\exp\left( -\frac{n}{2} (\lambda_n -1)^2 \right)\right).
\]
For example, if we choose $a_n$ to decrease as $\log(n)^{-1}$,
then the above shows that remote contiguity limits the perturbation
to be of smaller order than $\sqrt{\log(\log(n))/n}$.

Unfortunately, remotely contiguous domains of attraction for
near-critical maximal cluster sizes (for example, those intended
in equation~(\ref{eq:RCcritCmax})) have an extent of order
$(n^{-1}\log(n))^{1/2}$, not the order $n^{-1/3}$ that
occurs in theorem~\ref{thm:ERcrit}. So remote contiguity
does not cover the entire range of possible perturbations
that preserve near-critical maximal cluster sizes. If we
impose perturbations proportional to $n^{-1/3}$,
requiring remote contiguity leads to exponential rates
$a_n\sim \exp(-n^{1/3})$, which overwhelms the polynomial
factor in assertion (\ref{eq:ERcritmaxcluster}).

This illustrates a limitation that is important to point
out: remote contiguity makes no distinction between asymptotic
assertions, other than by rate: as long as probabilities
$P_n(A_n)$ converge to zero fast enough, \cf\
(\ref{eq:defrc}), remote contiguity asserts
$Q_n(A_n)=o(1)$, without regard for the further
details involved in the definition of the events $A_n$. 
In the case at hand, when we ask questions regarding the
size of the maximal cluster, there are properties very
specific to homogeneous \ER\ graphs at
criticality, that enable $n^{-1/3}$-proportionality of
the critical window. Lemmas~\ref{lem:testnorc}
and~\ref{lem:aff} demonstrate  that there are other
asymptotic assertions $(B_n)$ with probabilities
$P_{1,n}(B_n)$ of order $o(a_n)$, but with
probabilities $P_{1+O(n^{-1/3}),n}(B_n)$ that do not
go to zero.

\subsection{Asymptotic connectedness in \ER\ graphs}

Recall that any homogeneous \ER\ graph with edge
probability $\lambda_n/n $ is disconnected with high
probability if $\limsup_{n \to \infty }\lambda_n <\infty$
(see \eg\ section~5.3 of \citep{hofstad16}). The results
of this subsection apply to \ER\ graphs with diverging
$(\lambda_n)$, typically of $O(\log(n))$.
\begin{lemma}
	\label{lem:conncect}
	Let $\lambda_n \to \infty$ as $n \to \infty$. If
	$\lambda_n-\log (n) \to -\infty$, then,
	\[
	P_{\lambda_n,n}\left(\scrC_{max}\,\,\text{\rm is connected}
	\right)=O\left(\frac{\lambda_n}{n-\lambda_n}\right)=o(1).
	\]
	If, instead, $\lambda_n-\log (n) \to \infty$, then, 
	\[
	P_{\lambda_n,n}\left(\scrC_{max}\,\,\text{\rm is disconnected}\right)
	=O\left({n^{-1/4}}\right).
	\]
\end{lemma}
\begin{proof}
	The first result is a direct consequence of the first
	inequality in proposition~5.10 and equations~(5.3.25)--(5.3.26)
	in \citep{hofstad16}. As for the second result, with
	$\lambda_n^*=\min(\lambda_n, 2 \log (n))$,
	\begin{equation}
		P_{\lambda_n,n}\left(\scrC_{max}\text{ is disconnected}\right)
		\leq 1-P_{\lambda_n^*,n}\left(\scrC_{max}\text{ is connected}\right) .
	\end{equation}
	Then the conclusion follows by using equations~(5.3.14),
	(5.3.21)--(5.3.24) and~(5.3.27) in \citep{hofstad16} with
	$\lambda=\lambda_n^*$.
\end{proof}
\begin{example}{\it (Connectivity in inhomogeneous \ER\ graphs)}\\
	Consider an inhomogeneous {\ER} graph with edge probabilities
	$q_{n,ij}= c_{n,ij}\log(n)/n$. Then, a sufficient condition for
	such a graph to be asymptotically connected is the existence of
	a suitable sequence $(d_n)$, with $d_n>0$, $\liminf_{n \to \infty}d_n>1$ and
	$\lim_{n \to \infty} d_n \log(n)/n <1$. To see this, also assume that
	\[
	\sup_{i<j}\left|\frac{c_{n,ij}}{d_n}-1\right| \to 0,
	\quad  \limsup_{n \to \infty}
	\frac{\sum_{i<j}(c_{n,ij}-d_n)^2}{d_n(n- \log (n))}<\frac{1}{4}.
	\]
	Then the assumptions  $r_n=o(1)$ and $a_n=o(\exp(-R_n))$ of
	lemma~\ref{lem:ER} are satisfied, with $\mu_{n,ij}=d_n \log(n)$
	and $a_n=n^{-\delta}$, for some $\delta <1/4$. Hence
	$Q_n\contig n^{\delta}P_n$, where $P_n$ is the distribution of
	the homogeneous {\ER} graph with edge-probability $d_n \log(n)/n$.
	By lemma \ref{lem:conncect}, the latter has probability of not
	being connected of order $O(n^{-1/4})$, thus entailing that
	\[
	Q_n\left(\scrC_{max}\text{ is disconnected}\right)=o(1)
	\] 
	as $n\to \infty$ by remote contiguity. 
\end{example}


\section{Conclusions and discussion}
\label{sec:concdisc}

In what precedes, we have attempted to highlight how remote contiguity
can be used to generalize asymptotic properties, much like asymptotic
equivalence and contiguity, but with a wider range of applicability.
Particularly, we have shown that remote contiguity can be applied to
the connectivity properties of \ER\ graphs, in various regimes of edge
sparsity. Conditions are formulated for the defining parameters of the
random graph enabling remote contiguity and the generalization of 
asymptotic properties.

It is expected that remote contiguity proves helpful for the
generalization of other asymptotic random graph properties. For
example, it is known that the degree sequence of \ER\ graphs
distributed according to $P_{\lambda,n}$, for some $\lambda>0$,
converges to a Poisson distribution. Below we denote 
$p_k= e^{-\lambda}\lambda^k/k!$, ($k\geq1$), and 
$P_k^{(n)}(X^n)= n^{-1} \sum_{i=1}^n 1_{D_i(X^n)=k}$ for the
empirical degree distribution. 
\begin{proposition}
	For any $\lambda>0$ and \ER\ graphs $X^n\sim P_{\lambda,n}$,
	we have
	\begin{equation}
		\label{eq:conv1}
		P_{\lambda,n}\left( \max_{k\geq 1} |P_{k}^{(n)}(X_n)-p_k|>\varepsilon_n
		\right)=O(1/(n\varepsilon_n^2))
	\end{equation}
	as $n \to \infty$, for any $\varepsilon_n \downarrow 0$ such that
	$n \varepsilon_n \to \infty$.
\end{proposition}
\begin{proof}
	The result in \eqref{eq:conv1} follows immediately from the inequality
	\[
	\sum_{k \geq 0} |\EE_{P_{\lambda,n}}(P_k^{(n)}(X^n))- p_k| < 2\varepsilon_n,
	\] 
	valid for all large $n$, and equation~(5.4.17) in \citep{hofstad16}.
\end{proof}
Application of remote contiguity enables the following generalization.
\begin{corollary}
	For any \ER\ graph $Y^n$ of law $Q_n=P_{(\lambda_{n,ij}),n}$
	satisfying $\sup_{i<j}|\lambda_{n,ij}-\lambda|=o(1)$ and
	$\sum_{i<j}|\lambda_{n,ij}-\lambda|^2< n\lambda\delta\log(n\varepsilon_n^2)$,
	with $0<\delta<1$, as $n \to \infty$ we have
	\begin{equation}
		\label{eq:conv2}
		Q_n\left(
		\max_{k\geq 1} |P_{k}^{(n)}(Y_n)-p_k|>\varepsilon_n\right)=o(1).
	\end{equation}
\end{corollary}
\begin{proof}
	The result in \eqref{eq:conv2} immediately follows from \eqref{eq:conv1}
	and an application of corollary \ref{cor:inh}.
\end{proof}
More ambitious forms of generalization are conceivable: for example, one
could consider the so-called preferential attachment graph, which displays
a degree distribution with heavy tails asymptotically (see, \cite{hofstad16},
theorem~8.3). Dependence of edges makes the analysis more demanding
technically, but the machinery of remote contiguity continues to apply.
Thus one may study the extent to which the model of (\cite{hofstad16},
equation~(8.2.1)) may be perturbed, without influencing the asymptotic
tail behaviour of the degree distribution.

But the application of remote contiguity is not limited to random
graphs; generalization of \emph{any} asymptotic property in \emph{any}
sequence of probabilistic models can be analysed with remote
contiguity. To illustrate this, we note that for two sequences
$(P_n)$ and $(Q_n)$ on measurable spaces $(\scrX_n,\scrB_n)$,
we have $Q_n\contig a^{-1}_n P_n$, if for every $\ep>0$ there
exists a $\delta>0$ such that
\[
Q_n\Bigl(\frac{dP_n}{dQ_n}<\delta a_n\Bigr)<\ep,
\] 
(or, equivalently, if every subsequence of $(a_n(dP_n/dQ_n)^{-1})$
has a weakly converging subsequence). Lemma~\ref{lem:rcfirstlemma}
gives a variety of general conditions to establish remote
contiguity, analogous to \emph{Le~Cam's First Lemma}
(\cite{LeCam86}, chapter~3, section~3, proposition~3).
Moreover, the arguments of subsection~\ref{sub:IERRCnecc}
are fully general, so there are also general conditions to
exclude remote contiguity, for example, if the Hellinger affinity
decreases to zero fast enough:
\[
\alpha(P_n,Q_n) = o\bigl( a_n^{1/2} \bigr).
\]
We therefore express the hope that remote contiguity can be applied
in more general examples besides random graphs, in a role
that generalizes the role of contiguity.


\appendix

\section{Remote contiguity}
\label{app:rc} 

Remote contiguity was introduced
in \citep{Kleijn21} to demonstrate that asymptotic properties of
Bayesian posterior distributions can be lifted to frequentist
statements of asymptotic consistency, hypothesis testing, model
selection and uncertainty quantification. For another statistical
example in the setting of extreme-value theory,
\citep{Falk20,Padoan22} use remote contiguity to
prove consistency with respect to relatively complicated
true data distributions by simpler, approximating sequences of
max-stable distributions.

Here and elsewhere, $M^1(\scrX)$ denotes the collection of all
probability measures on a measurable space $(\scrX,\scrB)$.
\begin{definition}
	\label{def:remctg}
	Given measurable spaces $(\scrX_n,\scrB_n)$ with two
	sequences of probability measures $P_n,Q_n\in M^1(\scrX_n)$
	for all $n\geq1$, and a sequence $\rho_n\downarrow0$, we say that
	$Q_n$ is $\rho_n$-remotely contiguous with respect to $P_n$,
	notation $Q_n\contig \rho_n^{-1}P_n$, if,
	\[
	P_n\phi_n(X^n) = o(\rho_n)
	\quad\Rightarrow\quad Q_n\phi_n(X^n)=o(1),
	\]
	for every sequence of $\scrB_n$-measurable $\phi_n:\scrX_n\rightarrow[0,1]$.
\end{definition}
Given two sequences $(P_n)$ and $(Q_n)$, contiguity $P_n\ctg Q_n$ is
equivalent to remote contiguity $P_n\ctg a_n^{-1} Q_n$ for all
$a_n\downarrow0$.

The following is the remotely contiguous analogue of Le~Cam's
First Lemma (\cite{LeCam86}, section~3.3, proposition 3).
\begin{lemma}
	\label{lem:rcfirstlemma}
	Let probability measures $(P_n)$, $(Q_n)$ on measurable spaces
	$(\scrX_n,\scrB_n)$ and $a_n\downarrow0$ be given, then
	$Q_n\contig a_n^{-1}P_n$ if any of the following hold:
	\begin{itemize}
		\item[(i)] for any bounded, $\scrB_n$-msb.
		$T_n:\scrX_n\rightarrow[0,1]$,
		$a_n^{-1}T_n\conv{P_n}0$ $\Rightarrow$ $T_n\convprob{Q_n}0$,
		\item[(ii)] for any $\ep>0$, there is a $\delta>0$ such that
		$Q_n(dP_n/dQ_n<\delta\,a_n)<\ep$, for large enough $n$,
		\item[(iii)] there is a $b>0$ such that
		$\liminf_{n} b\,a_n^{-1}P_n(dQ_n/dP_n>b\,a_n^{-1})=1$,
		\item[(iv)] for any $\ep>0$, there is a constant $c>0$ such that
		$\|Q_n-Q_n\wedge c\,a_n^{-1}P_n\|<\ep$, for large enough $n$,
		\item[(v)] 
		under $Q_n$ every subsequence of $(a_n(dP_n/dQ_n)^{-1})$
		has a weakly convergent subsequence.
	\end{itemize}
\end{lemma}
\begin{remark}
	For any measurable space $(\scrX,\scrB)$, the definition of
	$(dP/dQ)^{-1}:\scrX\to(0,\infty]:x\mapsto 1/(dP/dQ(x))$ is
	$Q_n$-almost-sure: given a (sigma-finite) measure $\nu$ that
	dominates both $P$ and $Q$ (\eg\ $\nu=P+Q$), denote
	$dP/d\nu=p$ and $dQ/d\nu=q$. Then the measurable map
	$p/q\,1\{q>0\}:\scrX\rightarrow[0,\infty)$ is a
	$\nu$-almost-everywhere version of $dP/dQ$, and
	$q/p\,1\{q>0\}:\scrX\rightarrow[0,\infty]$ defines
	$(dP/dQ)^{-1}$ $Q$-almost-surely.
\end{remark}
Characterization
{\it (v)} provides the most insightful formulation, relating
remote contiguity to weak convergence of re-scaled likelihood ratios,
\cf\ \citep{LeCam86}.
In most applications, characterization {\it (ii)} is the most
practical to demonstrate remote contiguity.


\section{The Lindeberg-Feller theorem}
\label{app:weakconvergence}

In its most basic form, the Lindeberg-Feller theorem formulates a condition
for the convergence of sums of independent (but not necessarily
identically distributed) random variables to a central limit. There exist
versions for dependent random variables too. A \emph{triangular array}
consists of a sequence $(k(n))$ that increases to infinity, and
random variables $X_{n,k}$, where $n\geq1$ and $1\leq k\leq k(n)$.
\begin{theorem}
	\label{thm:lindebergfeller}
	For each $n\geq1$, let $X_{n,k}$, $1\leq k\leq k(n)$, be independent, with
	expectations $\mu_{n,k}\in\RR$ and variances $\sigma_{n,k}^2<\infty$.
	With $s_n^2=\sum_{k=1}^{k(n)}\sigma_{n,k}^2$, assume that, for every
	$\ep>0$,
	\[
	\frac{1}{s_n^2}\sum_{k=1}^{k(n)}
	\EE \|X_{n,k}-\mu_{n,k}\|^2 1_{\{\|X_{n,k}-\mu_{n,k}\|>\ep s_n\}}\to 0.
	\]
	Then $s_n$-normalized, $\mu_{n,k}$-centred sums converge weakly to the
	standard normal distribution,   
	\[
	\frac{1}{s_n}\sum_{k=1}^{k(n)} \bigl( X_{n,k} - \mu_{n,k} \bigr)
	\convweak{}N(0,1).
	\]
\end{theorem}

\bibliographystyle{chicago} 
\bibliography{rcrg1}

\end{document}